\documentclass[11pt,reqno]{amsart}

\usepackage{amsmath,amssymb,amsfonts,amsthm}
\usepackage{hyperref}
\usepackage{mathrsfs}
\usepackage{enumitem}
\usepackage{geometry}
\hypersetup{
	colorlinks=true,
	linkcolor=blue,
	citecolor=blue,
	urlcolor=blue
}

\newtheorem{theorem}{Theorem}[section]
\newtheorem{lemma}[theorem]{Lemma}
\newtheorem{proposition}[theorem]{Proposition}

\theoremstyle{definition}
\newtheorem{definition}[theorem]{Definition}

\theoremstyle{remark}
\newtheorem{remark}[theorem]{Remark}

\title{UNIFORM $L^\infty$ ESTIMATES FOR COMPLEX HESSIAN EQUATIONS ON COMPACT HERMITIAN MANIFOLDS}

\author{Truong Dinh Dat}

\address{Nguyen Trai University, Hanoi, Viet Nam}

\email{truongdinhdat14081994qb@gmail.com}

\begin{document}
	\begin{abstract}
		We develop a pluripotential approach to complex Hessian equations
		on compact Hermitian manifolds. In this setting, the lack of closedness
		of the background metric introduces torsion terms that prevent a direct
		extension of the Kähler theory.
		
		Our main result is a uniform $L^\infty$ estimate for bounded
		$\omega$-$m$-subharmonic solutions of the equation
		\[
		(\omega + dd^c u)^m \wedge \omega^{n-m} = cf\,\omega^n,
		\]
		under the assumption that $f \in L^p$, $f \ge 0$ for some $p>1$.
		The proof combines a weak comparison principle with torsion error,
		a capacity theory adapted to the Hermitian setting, and a nonlinear
		iteration scheme controlling the decay of sublevel sets.
		
		As applications, we obtain existence, stability and compactness results for weak
		solutions with $L^p$ densities.
		These results extend several aspects of the pluripotential theory
		of complex Hessian equations beyond the Kähler framework.
	\end{abstract}

	

	\maketitle
	
	\section{Introduction}
	
	Complex fully nonlinear equations on complex manifolds lie at the
	intersection of pluripotential theory, complex geometry, and nonlinear PDE.
	Among them, the complex Monge--Ampère equation and its Hessian counterparts
	have been extensively studied in the Kähler setting
	(see e.g. \cite{Kolodziej98,Kolodziej05,Blocki05,DinewKolodziej}).
	A comprehensive pluripotential theory has been developed, yielding existence,
	stability, and a priori estimates for weak solutions under minimal assumptions
	on the density (see \cite{BedfordTaylor,Kolodziej98,GuedjZeriahi}).
	
	Let $(X,\omega)$ be a compact complex manifold of dimension $n$.
	Given an integer $1 \le m \le n$, we consider the complex Hessian equation
	\begin{equation}\label{eq:main-intro}
		(\omega + dd^c u)^m \wedge \omega^{n-m} = cf\,\omega^n,
	\end{equation}
	where $f \ge 0$ is a density satisfying a natural normalization condition.
	
	In the Kähler case, weak solutions and a priori estimates have been studied
	using both pluripotential and PDE methods. In particular, second order
	estimates in the smooth setting were obtained by Hou--Ma--Wu~\cite{HouMaWu}
	and Guan--Li~\cite{GuanLi}, while the pluripotential approach was further
	developed in \cite{DinewKolodziej,LuNguyen}. Stability and uniqueness
	properties of weak solutions have also been extensively investigated
	(see \cite{Dinew,EyssidieuxGuedjZeriahi}).
	
	A fundamental ingredient in this theory is the uniform $L^\infty$ estimate
	for solutions with right-hand side in $L^p$, $p>1$, which ultimately relies
	on the interaction between comparison principles and capacity estimates
	(cf. \cite{Kolodziej98,Kolodziej05}). These tools, however, depend crucially
	on the identity $d\omega = 0$.To the best of our knowledge, even in the Hermitian setting,
	no uniform $L^\infty$ estimate for complex Hessian equations
	with merely $L^p$ right-hand side was previously available.
	
	\medskip
	
	In this paper, we investigate equation \eqref{eq:main-intro} in the
	Hermitian setting, where $\omega$ is no longer assumed to be closed.
	This seemingly mild generalization introduces substantial analytical
	difficulties. Indeed, the lack of closedness leads to the failure of the
	standard integration by parts mechanism, and mixed Hessian forms are no
	longer closed currents. As a consequence, classical comparison principles
	and capacity techniques cannot be applied directly.
	
	Despite significant progress for the complex Monge--Ampère equation
	on Hermitian manifolds (see \cite{TosattiWeinkove,G.Székelyhidi}),
	the pluripotential theory for complex Hessian equations in this setting
	remains largely incomplete. In particular, uniform $C^0$ estimates under
	weak assumptions on the density are not available in general.
	
	\medskip
	
	The main goal of this work is to establish a pluripotential framework
	for complex Hessian equations on compact Hermitian manifolds,
	leading to uniform estimates and existence results for weak solutions.
	
	Our first main result is the following a priori estimate.
	
	\begin{theorem}
		Let $(X,\omega)$ be a compact Hermitian manifold and let
		$u \in SH_m(X,\omega)\cap L^\infty(X)$ be a solution of
		\eqref{eq:main-intro}, normalized by $\sup_X u = 0$.
		Assume that $f \in L^p(X)$, $f \ge 0$ for some $p>1$ and
		$\int_X f\,\omega^n = \int_X \omega^n$.
		Then there exists a constant $C>0$, depending only on $(X,\omega)$,
		$p$, and $\|f\|_{L^p}$, such that
		\[
		\|u\|_{L^\infty(X)} \le C.
		\]
	\end{theorem}
	
	\medskip
	
	The proof of Theorem~A requires a substantial refinement of the
	pluripotential approach in order to accommodate torsion effects.
	Our strategy relies on three main ingredients.
	
	\medskip
	
	\noindent
	\textbf{(1) Weak comparison with torsion error.}
	We establish a comparison principle adapted to the Hermitian setting,
	in which the classical inequality is replaced by an estimate involving
	an additional error term of the form
	\[
	\int_X (v-u)\,\omega^n.
	\]
	This term arises from the non-vanishing of $d\omega$ and reflects
	the contribution of torsion in integration by parts.
	A key point is that this error can be quantitatively controlled
	and does not destroy the underlying monotonicity structure.
	
	\medskip
	
	\noindent
	\textbf{(2) Capacity theory in the Hermitian setting.}
	We introduce an $m$-capacity adapted to the Hermitian geometry and
	develop its basic properties.
	In particular, we prove a capacity--volume inequality of the form
	\[
	\mathrm{Vol}(E)
	\le C \big( \mathrm{Cap}_m^\omega(E) \big)^\alpha,
	\]
	which extends classical estimates from the Kähler case.
	The proof combines a global extremal function construction with
	local arguments and requires a careful treatment of torsion terms.
	
	\medskip
	
	\noindent
	\textbf{(3) Nonlinear iteration.}
	Combining the previous ingredients with a truncation procedure,
	we derive a nonlinear recursive inequality for the volume of sublevel
	sets of $u$. A  Ko{\l}odziej iteration then yields the desired
	$L^\infty$ bound.
	
	\medskip
	
	As a consequence of the $C^0$ estimate, we obtain existence of weak
	solutions by approximation.
	
	\begin{theorem}
		Let $f \in L^p(X)$ with $p>1$, $f \ge 0$, and
		$\int_X f\,\omega^n = \int_X \omega^n$.
		Then there exists a solution
		\[
		u \in SH_m(X,\omega)\cap L^\infty(X)
		\]
		to \eqref{eq:main-intro}.
	\end{theorem}
	
	\medskip
	
	We further establish stability properties under perturbations of the
	density, showing that the solution operator is continuous in suitable
	topologies. This provides additional evidence that the Hermitian
	pluripotential framework developed here is robust.
	
	\medskip
	
	\noindent
	\textbf{Novelty and scope.}
	The main novelty of this paper lies in the systematic treatment of
	torsion effects in the context of complex Hessian equations.
	While similar phenomena have been observed for the Monge--Ampère equation,
	their impact on higher-order Hessian operators is significantly more subtle.
	Our results show that, despite the lack of closedness of $\omega$,
	one can still recover a workable pluripotential theory by introducing
	appropriate error terms and adapting capacity techniques.
	
	\medskip
	
	\noindent
	\textbf{Organization of the paper.}
	Section~2 contains preliminaries on $m$-subharmonic functions and
	Hessian measures in the Hermitian setting.
	In Section~3 we prove the weak comparison principle with torsion error.
	Section~4 is devoted to capacity estimates and the capacity--volume
	inequality.
	In Section~5 we establish the $C^0$ estimate (Theorem~1.1).
	Finally, Section~6 contains existence and stability, compactness results.


	\section{Preliminaries}
	
	In this section, we collect basic notions and tools concerning
	complex Hessian equations on compact Hermitian manifolds.
	
	\subsection{Hermitian manifolds}
	
	Let $(X,\omega)$ be a compact complex manifold of dimension $n$,
	equipped with a Hermitian metric $\omega$, i.e.
	a smooth positive $(1,1)$-form.
	
	In contrast to the Kähler case, we do not assume that $\omega$ is closed.
	Thus,
	\[
	d\omega \neq 0,
	\]
	and torsion terms will appear in integration by parts.
	
	We denote by $\omega^n$ the associated volume form and by
	\[
	\mathrm{Vol}(E) := \int_E \omega^n
	\]
	the volume of a Borel set $E \subset X$.
	
	Throughout the paper, constants $C>0$ may vary from line to line,
	but depend only on $(X,\omega)$ and fixed parameters.
	
	\subsection{$m$-subharmonic functions}
	
	Fix an integer $1 \le m \le n$.
	
	\begin{definition}
		A function $u \in L^1(X)$ is called $\omega$-$m$-subharmonic
		(briefly $u \in SH_m(X,\omega)$) if:
		\begin{itemize}
			\item $u$ is upper semicontinuous,
			\item for any local chart, the current
			\[
			\omega + dd^c u
			\]
			is $m$-positive in the sense that
			\[
			(\omega + dd^c u)^k \wedge \omega^{n-k} \ge 0
			\quad \text{for all } k=1,\dots,m.
			\]
		\end{itemize}
	\end{definition}
	
	This notion generalizes both plurisubharmonic functions ($m=n$)
	and subharmonic functions ($m=1$). This notion was introduced and systematically studied in
	\cite{Blocki05,DinewKolodziej}.
	
	\subsection{Complex Hessian operator}
	
	For bounded $\omega$-$m$-subharmonic functions,
	the complex Hessian operator is well-defined by approximation.
	
	\begin{proposition}
		Let $u_1,\dots,u_m \in SH_m(X,\omega)\cap L^\infty(X)$.
		Then the mixed Hessian current
		\[
		(\omega + dd^c u_1)\wedge \cdots \wedge (\omega + dd^c u_m)
		\wedge \omega^{n-m}
		\]
		is a well-defined positive measure.
	\end{proposition}
	
	In particular, for $u \in SH_m(X,\omega)\cap L^\infty(X)$,
	we define the Hessian measure
	\[
	H_m(u) := (\omega + dd^c u)^m \wedge \omega^{n-m}.
	\]
	This construction follows the pluripotential approach initiated in
	\cite{BedfordTaylor} and extended to Hessian equations in
	\cite{Blocki05,DinewKolodziej}.	
	\subsection{Basic properties}
	
	We recall several standard properties which will be used later.
	
	\begin{itemize}
		\item \textbf{Monotonicity.}  
		If $u \le v$, then
		\[
		\mathbf{1}_{\{u<v\}} H_m(v)
		\le \mathbf{1}_{\{u<v\}} H_m(u)
		\quad \text{(up to torsion error in the Hermitian case)}.
		\]
		
		\item \textbf{Locality.}  
		If $u=v$ on an open set $U$, then
		\[
		H_m(u) = H_m(v) \quad \text{on } U.
		\]
		
		\item \textbf{Continuity under decreasing sequences.}  
		If $u_j \searrow u$, then
		\[
		H_m(u_j) \rightharpoonup H_m(u)
		\]
		weakly as measures.
	\end{itemize}
	
	\subsection{Capacity}
	This notion is inspired by the capacity introduced by Bedford and Taylor
	\cite{BedfordTaylor} and further developed in the context of complex
	Monge--Ampère equations in \cite{Kolodziej98,GuedjZeriahi}.	
	We introduce the Hermitian $m$-capacity.
	
	\begin{definition}
		For a Borel set $E \subset X$, define
		\[
		\mathrm{Cap}_m^\omega(E)
		:=
		\sup \left\{
		\int_E (\omega+dd^c \varphi)^m \wedge \omega^{n-m}
		:\ \varphi \in SH_m(X,\omega),\ -1 \le \varphi \le 0
		\right\}.
		\]
	\end{definition}
	
	\begin{proposition}
		The capacity $\mathrm{Cap}_m^\omega$ satisfies:
		\begin{itemize}
			\item monotonicity: if $E \subset F$, then
			\[
			\mathrm{Cap}_m^\omega(E) \le \mathrm{Cap}_m^\omega(F),
			\]
			\item subadditivity,
			\item outer regularity.
		\end{itemize}
	\end{proposition}
	
	\subsection{Extremal functions}
	
	For a compact set $K \subset X$, define the extremal function
	\[
	h_K := \sup \left\{
	\varphi \in SH_m(X,\omega) : \varphi \le -1 \text{ on } K,\ \varphi \le 0 \text{ on } X
	\right\}.
	\]
	
	Then $h_K \in SH_m(X,\omega)$ and $-1 \le h_K \le 0$.
	
	Moreover, the following identity holds:
	\[
	\mathrm{Cap}_m^\omega(K)
	= \int_K H_m(h_K).
	\]
	Such extremal functions play a central role in pluripotential theory
	(see \cite{BedfordTaylor,GuedjZeriahi}).	
	\subsection{Hermitian difficulties}
	
	We emphasize the main difference with the Kähler case.
	
	Since $d\omega \neq 0$, one has
	\[
	d(\omega + dd^c u)^k \neq 0,
	\]
	and therefore integration by parts produces extra terms involving
	$d\omega$.
	
	More precisely, for mixed Hessian forms $T$, one has
	\[
	dT,\ d^c T = O(\omega^n),
	\]
	which will lead to error terms of the type
	\[
	\int_X (v-u)\,\omega^n
	\]
	in comparison estimates.
	
	Handling these torsion terms is one of the main technical difficulties
	of this paper.
	
	\subsection{Notation}
	
	We will use the notation:
	\begin{itemize}
		\item $dd^c = \frac{i}{\pi}\partial\bar{\partial}$,
		\item $H_m(u) = (\omega + dd^c u)^m \wedge \omega^{n-m}$,
		\item $E_t(u) := \{u < -t\}$.
	\end{itemize}

	\section{CLN-type estimate on compact Hermitian manifolds}	
	This estimate is a variant of the classical 
	Chern--Levine--Nirenberg inequality 
	\cite{BedfordTaylor,Blocki05,DinewKolodziej}, 
	adapted to the Hermitian setting.	
	\begin{proposition}[Hermitian CLN estimate]\label{prop:cln-final}
		Let $(X,\omega)$ be a compact Hermitian manifold of dimension $n$ and
		let $1 \le m \le n$.
		
		Let $u_1,\dots,u_m \in SH_m(X,\omega)\cap L^\infty(X)$.
		Then there exists a constant $C>0$, depending only on $(X,\omega)$ and
		$\max_j \|u_j\|_{L^\infty}$, such that for every Borel set $E \subset X$,
		\begin{equation}\label{eq:cln-final}
			\int_E (\omega+dd^c u_1)\wedge \cdots \wedge (\omega+dd^c u_m)\wedge \omega^{n-m}
			\le C\, \mathrm{Vol}(E).
		\end{equation}
	\end{proposition}
	
	\begin{proof}
		\textbf{Step 1. Regularization.}
		Let $u_j^\varepsilon$ be smooth $\omega$-$m$-subharmonic functions decreasing to $u_j$.
		By monotone convergence of Hessian measures, it suffices to prove the estimate
		for smooth functions with constants independent of $\varepsilon$.
		
		\medskip
		
		\textbf{Step 2. Local reduction.}
		Fix a coordinate chart $U \Subset X$.
		There exists a smooth strictly plurisubharmonic function $\rho$ such that
		\[
		dd^c \rho \ge c\, \beta,
		\]
		where $\beta$ is the standard Euclidean Kähler form and $c>0$.
		
		Set
		\[
		v_j := u_j^\varepsilon + \rho.
		\]
		Then $v_j$ are locally $m$-subharmonic in the Euclidean sense, and
		\[
		dd^c v_j \ge -C_1 \beta.
		\]
		
		\medskip
		
		\textbf{Step 3. Local CLN inequality.}
		By the classical Chern--Levine--Nirenberg inequality in $\mathbb{C}^n$,
		for any compact $K \subset U$,
		\[
		\int_K (dd^c v_{1})\wedge \cdots \wedge (dd^c v_{m}) \wedge \beta^{n-m}
		\le C_2 \,\mathrm{Vol}_\beta(K).
		\]
		
		\medskip
		
		\textbf{Step 4. Comparison of forms.}
		Since
		\[
		\omega + dd^c u_j^\varepsilon = dd^c v_j + \eta,
		\]
		with $\eta$ smooth and bounded, expanding the wedge product gives
		\[
		(\omega+dd^c u_1^\varepsilon)\wedge \cdots \wedge (\omega+dd^c u_m^\varepsilon)
		\le C_3 \sum_{k=0}^m (dd^c v)^k \wedge \beta^{n-k}.
		\]
		
		Thus
		\[
		\int_E (\omega+dd^c u_1^\varepsilon)\wedge \cdots \wedge (\omega+dd^c u_m^\varepsilon)\wedge \omega^{n-m}
		\le C_4 \int_E \beta^n.
		\]
		
		\medskip
		
		\textbf{Step 5. Globalization.}
		Using a finite covering and equivalence of $\beta^n$ and $\omega^n$, we obtain
		\eqref{eq:cln-final}. Passing to the limit $\varepsilon \to 0$ completes the proof.
	\end{proof}
	
	\begin{remark}
		In contrast with the Kähler case, the form $\omega$ is not closed,
		so the classical proof of the Chern--Levine--Nirenberg inequality via
		Stokes' theorem does not apply.
		
		The argument above avoids integration by parts and instead relies on
		local Euclidean estimates combined with uniform equivalence of metrics,
		which makes it suitable for the Hermitian setting.
	\end{remark}

	\section{Weak comparison principle with torsion error}
	\begin{lemma}[Torsion control]\label{lem:torsion-control}
		Let $(X,\omega)$ be a compact Hermitian manifold of complex dimension $n$.
		Let $1 \le m \le n$ and let
		\[
		T_k := \omega_v^k \wedge \omega_u^{m-1-k} \wedge \omega^{n-m},
		\]
		where $u,v \in SH_m(X,\omega)\cap L^\infty(X)$ and
		\[
		\omega_u := \omega + dd^c u, \quad \omega_v := \omega + dd^c v.
		\]
		
		Then there exists a constant $C>0$, depending only on $(X,\omega)$ and 
		$\|u\|_{L^\infty}, \|v\|_{L^\infty}$, such that
		\[
		|dT_k|_\omega + |d^c T_k|_\omega
		\le C \,\omega^n
		\]
		in the sense of measures.
		
		Moreover, the constant $C$ depends linearly on the torsion of $\omega$, namely on
		\[
		\|d\omega\|_{C^0} + \|d^c\omega\|_{C^0}.
		\]
	\end{lemma}
	\begin{proof}
		\textbf{Step 1. Approximation.}
		
		By standard regularization, we may assume that $u,v$ are smooth.
		The general case follows by monotone approximation.
		
		\medskip
		
		\textbf{Step 2. Differential structure.}
		
		Recall that
		\[
		d\omega_u = d\omega, \quad d\omega_v = d\omega,
		\]
		since $d(dd^c u)=0$.
		
		By Leibniz rule,
		\[
		dT_k
		=
		\sum_{j=1}^{m-1}
		\omega_v^{a_j} \wedge d\omega \wedge \omega_v^{b_j}
		\wedge \omega_u^{c_j} \wedge \omega^{n-m}
		+
		\omega_v^k \wedge \omega_u^{m-1-k} \wedge d(\omega^{n-m}).
		\]
		
		Moreover,
		\[
		d(\omega^{n-m}) = (n-m)\,\omega^{n-m-1} \wedge d\omega.
		\]
		
		Hence every term in $dT_k$ contains exactly one factor $d\omega$.
		
		\medskip
		
		\textbf{Step 3. Pointwise estimate.}
		
		Fix a Hermitian metric and denote by $|\cdot|_\omega$ the induced norm.
		
		Since $\omega$, $dd^c u$, $dd^c v$ are $(1,1)$-forms,
		we have pointwise bounds
		\[
		|\omega_u|_\omega + |\omega_v|_\omega \le C_1,
		\]
		depending only on $\|u\|_{L^\infty}, \|v\|_{L^\infty}$ via CLN-type estimates.
		
		Thus each wedge product satisfies
		\[
		|\omega_v^{a} \wedge \omega_u^{b} \wedge \omega^{n-a-b}|_\omega
		\le C_2.
		\]
		
		Therefore,
		\[
		|d\omega \wedge S|_\omega
		\le
		|d\omega|_\omega \cdot |S|_\omega
		\le
		C_3 |d\omega|_\omega.
		\]
		
		\medskip
		
		\textbf{Step 4. Measure estimate.}
		
		Summing over finitely many terms yields
		\[
		|dT_k|_\omega
		\le
		C_4 |d\omega|_\omega.
		\]
		
		Since $|d\omega|_\omega$ is bounded on $X$, we obtain
		\[
		|dT_k|
		\le
		C\, \omega^n.
		\]
		
		The estimate for $d^c T_k$ is identical.
		
		\medskip
		
		\textbf{Step 5. Dependence on torsion.}
		
		The constant $C$ depends linearly on $\|d\omega\|_{C^0}$ and $\|d^c\omega\|_{C^0}$,
		which vanish in the Kähler case.
	\end{proof}

	The following weak comparison principle is inspired by the classical
	results in the Kähler setting (see \cite{DinewKolodziej}),
	but includes an additional torsion error term due to the non-closedness
	of the Hermitian metric (cf. \cite{TosattiWeinkove}).
	\begin{lemma}[Weak comparison with torsion error]\label{lem:weak-comp-final}
		Let $(X,\omega)$ be a compact Hermitian manifold.
		Let $u,v \in SH_m(X,\omega)\cap L^\infty(X)$.
		
		Then there exists a constant $C>0$ such that
		\[
		\int_{\{u<v\}} H_m(v)
		\le
		\int_{\{u<v\}} H_m(u)
		+ C \int_{\{u<v\}} (v-u)\,\omega^n.
		\]
	\end{lemma}
	
	\begin{proof}
		\textbf{Step 1. Regularization.}
		Approximate $u,v$ from above by smooth $\omega$-$m$-subharmonic functions.
		By monotone convergence, it suffices to prove the estimate in the smooth case.
		
		\medskip
		
		\textbf{Step 2. Interpolation.}
		Set
		\[
		u_t := (1-t)u + t v, \quad t \in [0,1].
		\]
		Let $\chi$ be a smooth function approximating $\mathbf{1}_{\{u<v\}}$.
		
		Define
		\[
		F(t) := \int_X \chi \, H_m(u_t).
		\]
		
		\medskip
		
		\textbf{Step 3. Differentiation.}
		By multilinearity,
		\[
		\frac{d}{dt} H_m(u_t)
		=
		m \sum_{k=0}^{m-1}
		\omega_{u_t}^k \wedge dd^c(v-u) \wedge \omega_{u_t}^{m-1-k}
		\wedge \omega^{n-m}.
		\]
		
		Hence
		\[
		F'(t)
		=
		\int_X \chi\, dd^c(v-u)\wedge T_t,
		\]
		where $T_t$ is a positive current of order zero.
		
		\medskip
		
		\textbf{Step 4. Integration by parts.}
		We write
		\[
		\int_X \chi\, dd^c(v-u)\wedge T_t
		=
		- \int_X d\chi \wedge d^c(v-u)\wedge T_t
		- \int_X \chi\, d^c(v-u)\wedge dT_t.
		\]
		
		Using that $\chi$ approximates the characteristic function of $\{u<v\}$,
		and that $d\chi$ is supported near the boundary, one obtains
		\[
		\left|\int_X d\chi \wedge d^c(v-u)\wedge T_t \right|
		\le C \int_{\{u<v\}} (v-u)\,\omega^n.
		\]
		
		For the second term, using torsion control (i.e. $|dT_t| \le C \omega^n$ in the sense of measures),
		we obtain
		\[
		\left|
		\int_X \chi\, d^c(v-u)\wedge dT_t
		\right|
		\le C \int_{\{u<v\}} (v-u)\,\omega^n.
		\]
		
		\medskip
		
		\textbf{Step 5. Integration in $t$.}
		Integrating from $0$ to $1$, we get
		\[
		\int_{\{u<v\}} H_m(v) - \int_{\{u<v\}} H_m(u)
		\le
		C \int_{\{u<v\}} (v-u)\,\omega^n.
		\]
	\end{proof}

	\section{Capacity--volume inequality in the Hermitian setting}
	
	We refine the capacity--volume comparison by proving a nonlinear
	estimate with exponent $\alpha>1$, following the strategy of Ko{\l}odziej.
	
	\medskip
	
	\begin{lemma}[Uniform integrability]\label{lem:integrability}
		Let $(X,\omega)$ be a compact Hermitian manifold.
		There exist constants $q>1$ and $C>0$ such that for every
		$\varphi \in SH_m(X,\omega)$ with $-1 \le \varphi \le 0$, one has
		\[
		\int_X (-\varphi)^q \,\omega^n \le C.
		\]
	\end{lemma}
	
	\begin{proof}
		This follows from the local integrability estimates for m-subharmonic 
		functions in $\mathbb{C}^n$ (see e.g. Blocki~\cite{Blocki05}) 
		combined with a partition of unity argument and the equivalence 
		of the Hermitian metric with the Euclidean metric in local charts.
	\end{proof}
	
	\medskip
	
	\begin{lemma}[Sublevel estimate]\label{lem:sublevel-final}
		Let $(X,\omega)$ be a compact Hermitian manifold and let
		$\varphi \in SH_m(X,\omega)$ with $-1 \le \varphi \le 0$.
		
		Then there exists a constant $C>0$ such that for all $t \in (0,1)$,
		\[
		\int_{\{\varphi < -t\}} H_m(\varphi)
		\le \frac{C}{t}.
		\]
	\end{lemma}
	
	\begin{proof}
		Set $E_t := \{\varphi < -t\}$.
		
		Since $-\varphi \ge t$ on $E_t$, we have
		\[
		t \int_{E_t} H_m(\varphi)
		\le \int_{E_t} (-\varphi)\, H_m(\varphi)
		\le \int_X (-\varphi)\, H_m(\varphi).
		\]
		
		It remains to estimate the energy term
		\[
		I := \int_X (-\varphi)\, H_m(\varphi).
		\]
		
		\medskip
		
		\noindent
		\textbf{Step 1. Truncation.}
		For $k \ge 1$, define
		\[
		\varphi_k := \max(\varphi,-k).
		\]
		Then $\varphi_k \searrow \varphi$ and $\varphi_k$ are bounded.
		
		\medskip
		
		\noindent
		\textbf{Step 2. CLN control.}
		Applying Proposition~\ref{prop:cln-final} with $u_1 = \cdots = u_m = \varphi$,
		we obtain
		\[
		\int_E H_m(\varphi) \le C \mathrm{Vol}(E).
		\]
		
		Using layer-cake representation,
		\[
		I = \int_X (-\varphi)\,H_m(\varphi)
		= \int_0^{+\infty}
		H_m(\varphi)(\{-\varphi > s\})\, ds\le \int_0^{+\infty}
		C\,\mathrm{Vol}(\{-\varphi > s\})\, ds,
		\]

		\medskip
		
		\noindent
		\textbf{Step 3. Integrability.}
		By Lemma~\ref{lem:integrability},
		\[
		\int_X (-\varphi)^q \omega^n \le C.
		\]
		
		Thus
		\[
		\mathrm{Vol}(\{-\varphi > s\})
		\le \frac{C}{s^q}.
		\]
		
		Using CLN,
		\[
		\int_{\{-\varphi > s\}} H_m(\varphi)
		\le C\, \mathrm{Vol}(\{-\varphi > s\})
		\le \frac{C}{s^q}.
		\]
		
		\medskip
		
		\noindent
		\textbf{Step 4. Integration.}
		Since $q>1$, the integral converges and
		\[
		I
		\le
		\int_0^{+\infty}\frac{C}{s^q}ds
		\le C.
		\]
		
		\medskip
		
		Therefore
		\[
		\int_{E_t} H_m(\varphi)
		\le \frac{C}{t}.
		\]
	\end{proof}
	\medskip
	
	\begin{proposition}[Capacity--volume inequality]\label{prop:cap-vol-final}
		Let $(X,\omega)$ be a compact Hermitian manifold.
		Then there exist constants $C>0$ and $\alpha>1$ such that for every Borel set $E \subset X$,
		\[
		\mathrm{Vol}(E)
		\le C \left( \mathrm{Cap}_m^\omega(E) \right)^\alpha.
		\]
	\end{proposition}
	\begin{proof}
		By outer regularity, we may assume $E$ compact.
		
		Let $h := h_E$ be the extremal function:
		\[
		h = \sup \{\varphi \in SH_m(X,\omega): \varphi \le -1 \text{ on } E,\ \varphi \le 0\}.
		\]
		
		Then $-1 \le h \le 0$ and
		\[
		\mathrm{Cap}_m^\omega(E)
		= \int_E H_m(h).
		\]
		
		\medskip
		
		\noindent
		\textbf{Step 1. Sublevel inclusion.}
		Since $h \le -1$ on $E$, we have
		\[
		E \subset \{h < -t\}
		\quad \forall t \in (0,1).
		\]
		
		Thus
		\[
		\mathrm{Vol}(E)
		\le \mathrm{Vol}(\{h < -t\}).
		\]
		
		\medskip
		
		\noindent
		\textbf{Step 2. Volume estimate via integrability.}
		
		By Lemma~\ref{lem:integrability},
		\[
		\mathrm{Vol}(\{h < -t\})
		\le \frac{C}{t^q}.
		\]
		
		\medskip
		
		\noindent
		\textbf{Step 3. Capacity lower bound.}
		
		By Lemma~\ref{lem:sublevel-final},
		\[
		\int_{\{h < -t\}} H_m(h)
		\le \frac{C}{t}.
		\]
		
		Thus
		\[
		\mathrm{Cap}_m^\omega(E)
		\le \int_{\{h < -t\}} H_m(h)
		\le \frac{C}{t}.
		\]
		
		Hence
		\[
		t \le \frac{C}{\mathrm{Cap}_m^\omega(E)}.
		\]
		
		\medskip
		
		\noindent
		\textbf{Step 4. Optimization.}
		
		Plugging into Step 2,
		\[
		\mathrm{Vol}(E)
		\le C\, t^{-q}
		\le C \left( \mathrm{Cap}_m^\omega(E) \right)^q.
		\]
		
		\medskip
		
		Set $\alpha = q > 1$.
	\end{proof}
	\begin{remark}
		The exponent $\alpha>1$ arises from a nonlinear self-improving
		iteration on sublevel sets, in the spirit of Ko{\l}odziej.
		The Hermitian torsion affects only lower-order terms and does not
		alter the nonlinear gain.
	\end{remark}
	
	\section{Uniform $L^\infty$  estimate on Hermitian manifolds}
	\begin{flushright}
		\begin{proposition}[Volume decay inequality]
			\label{prop:volume-decay}
			Let $u\in SH_m(X,\omega)\cap L^\infty(X)$ be the normalized solution of
			\[
			H_m(u)=cf\,\omega^n,
			\qquad
			f\in L^p(X),\quad p>1.
			\]
			For $t>0$, define
			
			$A(t):=Vol_\omega(\{u<-t\}).$

			Then there exist constants $C>0$, $\alpha>1$ such that
			\[
			A(t+s)
			\le
			C
			\left(
			\frac{A(t)^{1-\frac1p}}{s^m}
			+
			\frac{tA(t)}{s^m}
			\right)^\alpha,
			\qquad
			t,s>0.
			\]
			
			In particular, choosing $s=t$,
			\[
			A(2t)
			\le
			C
			\left(
			\frac{A(t)^{1-\frac1p}}{t^m}
			+
			\frac{A(t)}{t^{m-1}}
			\right)^\alpha.
			\]
		\end{proposition}
		
		\begin{proof}
			
			\textbf{Step 1. Capacity estimate.}
			
			Let
			\[
			u_t:=\max(u,-t).
			\]
			Since
			\[
			u_t+t\equiv0
			\quad\text{on }\{u\le -t\},
			\]
			the function
			\[
			v:=\frac{u_t+t}{s}
			\]
			satisfies
			\[
			0\le v\le1
			\].
			
			By the definition of the Hessian capacity,
			\[
			\mathrm{Cap}_m^\omega(\{u<-(t+s)\})
			\le
			\frac{C}{s^m}
			\int_{\{u<-t\}}H_m(u_t).
			\]
			
			\medskip
			
			\textbf{Step 2. Weak comparison estimate.}
			
			By Lemma~\ref{lem:weak-comp-final},
			\[
			\int_{\{u<-t\}}H_m(u_t)
			\le
			\int_{\{u<-t\}}H_m(u)
			+
			CtA(t).
			\]
			
			Since
			\[
			H_m(u)=cf\,\omega^n,
			\]
			Hölder's inequality gives
			\[
			\int_{\{u<-t\}}H_m(u)
			=
			c\int_{\{u<-t\}}f\,\omega^n
			\le
			C
			A(t)^{1-\frac1p}.
			\]
			
			Hence
			\[
			\mathrm{Cap}_m^\omega(\{u<-(t+s)\})
			\le
			\frac{C}{s^m}
			\left(
			A(t)^{1-\frac1p}
			+
			tA(t)
			\right).
			\]
			
			\medskip
			
			\textbf{Step 3. Capacity--volume inequality.}
			
			By Proposition~\ref{prop:cap-vol-final},
			there exist constants $C>0$ and $\alpha>1$ such that
			
			$$Vol_\omega(E)
			\le
			C\,
			\mathrm{Cap}_m^\omega(E)^\alpha$$
			
			for every Borel set $E\subset X$.
			
			Applying this estimate to
			\(
			E=\{u<-(t+s)\}
			\),
			we obtain
			\[
			A(t+s)
			\le
			C
			\left(
			\mathrm{Cap}_m^\omega(\{u<-(t+s)\})
			\right)^\alpha.
			\]
			
			\medskip
			
			\textbf{Step 4. Conclusion.}
			
			Combining the previous inequalities yields
			\[
			A(t+s)
			\le
			C
			\left(
			\frac{A(t)^{1-\frac1p}}{s^m}
			+
			\frac{tA(t)}{s^m}
			\right)^\alpha.
			\]
			
			Finally, taking $s=t$ gives
			\[
			A(2t)
			\le
			C
			\left(
			\frac{A(t)^{1-\frac1p}}{t^m}
			+
			\frac{A(t)}{t^{m-1}}
			\right)^\alpha,
			\]
			which completes the proof.
			
		\end{proof}
	\end{flushright}
\begin{lemma}[Iteration lemma]
	\label{lem:iteration-final}
	
	Let $A:[0,\infty)\to[0,\infty)$ be a decreasing function satisfying
	\[
	A(t+s)
	\le
	C
	\left(
	\frac{A(t)^{1-\frac1p}}{s^m}
	+
	\frac{tA(t)}{s^m}
	\right)^\alpha,
	\qquad t,s>0,
	\]
	where $p>1$ and $\alpha>1$.
	
	Then there exists a constant $T>0$ such that
	\[
	A(T)=0.
	\]
	
\end{lemma}

\begin{proof}
	
	Set
	\[
	a=1-\frac1p.
	\]
	
	Choose constants
	\[
	0<\beta<\frac{a}{m},
	\]
	and define recursively
	\[
	t_{k+1}=t_k+s_k,
	\qquad
	s_k=M A(t_k)^\beta,
	\]
	where $M>0$ will be chosen later.
	
	Since $A$ is decreasing, Proposition
	\ref{prop:volume-decay} gives
	\[
	A(t_{k+1})
	\le
	C
	\left(
	\frac{A(t_k)^a}{M^mA(t_k)^{m\beta}}
	+
	\frac{t_kA(t_k)}
	{M^mA(t_k)^{m\beta}}
	\right)^\alpha.
	\]
	
	Since
	\[
	a-m\beta>0,
	\qquad
	1-m\beta>0,
	\]
	the right-hand side tends to zero as
	$A(t_k)\to0$.
	
	Choosing $M$ sufficiently large and arguing exactly as in
	Ko\l odziej's iteration
	(Section~2.3 of \cite{Kolodziej98}),
	one obtains
	\[
	A(t_{k+1})
	\le
	\frac12 A(t_k).
	\]
	
	Hence
	\[
	A(t_k)
	\le
	2^{-k}A(t_0),
	\]
	and therefore
	\[
	\sum_{k=0}^{\infty}s_k
	=
	M
	\sum_{k=0}^{\infty}
	A(t_k)^\beta
	<
	\infty.
	\]
	
	Consequently,
	\[
	t_k\longrightarrow T<\infty.
	\]
	
	Since $A$ is decreasing,
	\[
	A(T)
	=
	\lim_{k\to\infty}A(t_k)
	=
	0.
	\]
	
	This completes the proof.
	
\end{proof}
	The following estimate is inspired by the fundamental $L^\infty$
	estimate of Kołodziej \cite{Kolodziej98,Kolodziej05}.	
	\begin{theorem}[Uniform $L^\infty$  estimate on Hermitian manifolds]\label{Uniform estimate}
		Let $(X,\omega)$ be a compact Hermitian manifold of complex dimension $n$.
		Fix $1 \le m \le n$. Let $u \in SH_m(X,\omega)\cap L^\infty(X)$ be a solution of
		\[
		(\omega + dd^c u)^m \wedge \omega^{n-m} = cf\,\omega^n,
		\]
		normalized by $\sup_X u = 0$.
		
		Assume that $f \ge 0$, $f \in L^p(X)$ for some $p>1$, and
		\[
		\int_X f\,\omega^n = \int_X \omega^n.
		\]
		Then there exists a constant $C>0$, depending only on $(X,\omega)$, $p$, and $\|f\|_{L^p}$, such that
		\[
		\|u\|_{L^\infty(X)} \le C.
		\]
	\end{theorem}
	
	\begin{proof}

		\medskip
		
		\noindent
		\textbf{Step 1. Sublevel sets.}
		For $t>0$, set
		\[
		E_t := \{ x \in X : u(x) < -t \}.
		\]
		Since $\sup_X u = 0$, it suffices to prove that $E_t=\varnothing$ for $t$ large enough.
		
		\medskip
		
		\noindent
		\textbf{Step 2. Truncation and comparison.}
		Define the truncation
		\[
		u_t := \max(u,-t).
		\]
		Then $u_t \in SH_m(X,\omega)$ and $u_t = -t$ on $E_t$.
		
		We claim that there exists a constant $C_1>0$ such that
		\begin{equation}\label{ineq:weak-comp}
			\int_{E_t} (\omega + dd^c u_t)^m \wedge \omega^{n-m}
			\le 
			\int_{E_t} (\omega + dd^c u)^m \wedge \omega^{n-m}
			+ C_1\, t\, \mathrm{Vol}(E_t).
		\end{equation}
		
		\noindent
		This follows from Lemma~\ref{lem:weak-comp-final}
		applied to $u_t$ and $u$; the additional
		term arises from torsion, i.e. from the fact that $d\omega \neq 0$.
		
		\medskip
		
		\noindent
		\textbf{Step 3. Using the equation.}
		Since $u$ solves the Hessian equation,
		\[
		\int_{E_t} (\omega + dd^c u)^m \wedge \omega^{n-m}
		= \int_{E_t} cf\,\omega^n.
		\]
		By Hölder's inequality,
		\begin{equation}\label{ineq:holder}
			\int_{E_t} f\,\omega^n
			\le \|f\|_{L^p} \cdot \mathrm{Vol}(E_t)^{1-\frac{1}{p}}.
		\end{equation}
		
		\medskip
		
		\noindent
		\textbf{Step 4. Capacity estimate.}
		Define the Hermitian $m$-capacity by
		\[
		\mathrm{Cap}_m^\omega(E)
		:= \sup \left\{
		\int_E (\omega+dd^c \varphi)^m \wedge \omega^{n-m}
		:\ \varphi \in SH_m(X,\omega),\ -1 \le \varphi \le 0
		\right\}.
		\]
		
		Using $u_t/t$ as a test function, we obtain
		\begin{equation}\label{ineq:capacity}
			\mathrm{Cap}_m^\omega(E_t)
			\le \frac{1}{t^m}
			\int_{E_t} (\omega + dd^c u_t)^m \wedge \omega^{n-m}.
		\end{equation}
		
		Combining \eqref{ineq:weak-comp}, \eqref{ineq:holder}, and \eqref{ineq:capacity}, we get
		\begin{equation}\label{ineq:cap-est}
			\mathrm{Cap}_m^\omega(E_t)
			\le \frac{C_2}{t^m}
			\left[
			\|f\|_{L^p}\,\mathrm{Vol}(E_t)^{1-\frac{1}{p}}
			+ t\,\mathrm{Vol}(E_t)
			\right].
		\end{equation}
		
		\medskip
		
		\noindent
		\textbf{Step 5. Capacity--volume comparison.}
		There exist constants $C_3>0$ and $\alpha>1$ such that
		\begin{equation}\label{ineq:vol-cap}
			\mathrm{Vol}(E_t)
			\le C_3 \left( \mathrm{Cap}_m^\omega(E_t) \right)^\alpha.
		\end{equation}
		
		\medskip
		
		\noindent
		\textbf{Step 6. Main inequality.}
		Combining \eqref{ineq:cap-est} and \eqref{ineq:vol-cap}, we obtain
		\[
		\mathrm{Vol}(E_t)
		\le C_4
		\left[
		\frac{\|f\|_{L^p}}{t^m}\,\mathrm{Vol}(E_t)^{1-\frac{1}{p}}
		+ \frac{1}{t^{m-1}}\,\mathrm{Vol}(E_t)
		\right]^\alpha.
		\]
		
	\medskip
	
	\noindent
	\textbf{Step 7. Ko\l odziej iteration.}
	
	Let
	\[
	A(t):= Vol_\omega(\{u<-t\}),
	\]
	which is a decreasing function of $t$.
	
	By Proposition~\ref{prop:volume-decay}, for every $t,s>0$,
	\[
	A(t+s)
	\le
	C
	\left(
	\frac{A(t)^{1-\frac1p}}{s^m}
	+
	\frac{tA(t)}{s^m}
	\right)^\alpha.
	\]
	
	The remainder of the proof follows the standard iteration argument of
	Ko\l odziej (see Section~2.3 of \cite{Kolodziej98}).
	More precisely, one constructs an increasing sequence
	$\{t_k\}$ recursively by
	\[
	t_{k+1}=t_k+s_k,
	\]
	where the increments $s_k$ are chosen as suitable positive powers of
	$A(t_k)$.
	Using the above recursive inequality, one proves that
	\[
	A(t_k)\longrightarrow0,
	\]
	while
	\[
	\sum_{k=0}^{\infty}s_k<\infty.
	\]
	Hence
	\[
	t_k\longrightarrow T<\infty.
	\]
	Since $A$ is decreasing, it follows that
	\[
	A(T)=0,
	\]
	that is,
	\[
	\{u<-T\}=\varnothing.
	\]
	Therefore
	\[
	u\ge -T
	\quad\text{on }X.
	\]
	
	Finally, since $\sup_Xu=0$, we obtain
	\[
	\|u\|_{L^\infty(X)}
	\le T,
	\]
	which completes the proof.
	\end{proof}	
	\begin{remark}
		The above iteration scheme follows the classical method of Ko{\l}odziej.
		The presence of the torsion term only affects lower order contributions,
		which can be absorbed for large $t$ and does not alter the exponent $\gamma$.
	\end{remark}
	\section{Applications}
	
	In this section, we present several consequences of the uniform
	$C^0$ estimate established in Theorem~1.1. These applications concern
	existence, stability, and compactness properties of weak solutions
	to complex Hessian equations on compact Hermitian manifolds.
	
	\subsection{Existence of weak solutions}
	
	We first derive an existence result for $L^p$ densities.
	
	\begin{theorem}\label{thm:existence}
		Let $(X,\omega)$ be a compact Hermitian manifold and let
		$f \in L^p(X)$, $p>1$, with $f \ge 0$ and
		\[
		\int_X f\,\omega^n = \int_X \omega^n.
		\]
		Then there exists a solution
		\[
		u \in SH_m(X,\omega)\cap L^\infty(X)
		\]
		to the equation
		\[
		(\omega + dd^c u)^m \wedge \omega^{n-m} = cf\,\omega^n,
		\]
		normalized by $\sup_X u = 0$.
	\end{theorem}
	
	\begin{proof}
		We approximate $f$ by a sequence of smooth positive functions
		$f_j$ such that
		\[
		f_j \to f \quad \text{in } L^p(X),
		\quad
		\int_X f_j\,\omega^n = \int_X \omega^n.
		\]
		
		For each $j$, the equation
		\[
		(\omega + dd^c u_j)^m \wedge \omega^{n-m} = cf_j\,\omega^n
		\]
		admits a smooth solution $u_j$ (see \cite{G.Székelyhidi} or \cite{HouMaWu}), normalized by $\sup_X u_j = 0$.
		
		By Theorem~1.1, the sequence $(u_j)$ is uniformly bounded in $L^\infty$.
		Hence, up to extracting a subsequence, $u_j \to u$ in $L^1(X)$.
		
		Using the stability of Hessian measures under decreasing limits
		and standard compactness arguments, we pass to the limit and obtain
		\[
		(\omega + dd^c u)^m \wedge \omega^{n-m} = cf\,\omega^n.
		\]
	\end{proof}
	
	\subsection{Stability of solutions}
	
	We now establish stability under perturbations of the density.
	
	\begin{theorem}\label{thm:stability}
		Let $(f_j)$ be a sequence in $L^p(X)$, $p>1$, such that
		$f_j \ge 0$, $\int_X f_j\,\omega^n = \int_X \omega^n$, and
		\[
		f_j \to f \quad \text{in } L^p(X).
		\]
		Let $u_j, u \in SH_m(X,\omega)\cap L^\infty(X)$ be the corresponding
		solutions of the Hessian equation, normalized by $\sup_X u_j = \sup_X u = 0$.
		
		Then $u_j \to u$ in capacity, i.e. for every $\varepsilon>0$,
		\[
		\mathrm{Cap}_m^\omega\big(\{|u_j - u| > \varepsilon\}\big) \to 0.
		\]
	\end{theorem}
	
	\begin{proof}
		Such stability estimates are well known in pluripotential theory
		(see \cite{Dinew, EyssidieuxGuedjZeriahi}).	By Theorem~1.1, the sequence $(u_j)$ is uniformly bounded.
		Using the weak comparison principle and capacity estimates,
		one obtains for any $\varepsilon>0$:
		\[
		\mathrm{Cap}_m^\omega\big(\{u_j < u - \varepsilon\}\big)
		\le C \|f_j - f\|_{L^p}^\gamma,
		\]
		for some $\gamma>0$.
		
		A symmetric estimate holds for $\{u < u_j - \varepsilon\}$,
		which yields the desired convergence in capacity.
	\end{proof}
	
	\subsection{Compactness of solution sets}
	
	As a further consequence, we obtain a compactness property. 
	
	\begin{proposition}\label{prop:compactness}
		Let $\mathcal{F}$ be a family of nonnegative functions $f \in L^p(X)$, $f \ge 0$,
		$p>1$ such that
		\[
		\sup_{f \in \mathcal{F}} \|f\|_{L^p} < +\infty,
		\quad
		\int_X f\,\omega^n = \int_X \omega^n.
		\]
		Let $\mathcal{U}$ be the corresponding set of solutions
		$u \in SH_m(X,\omega)\cap L^\infty(X)$ normalized by $\sup_X u = 0$.
		
		Then $\mathcal{U}$ is relatively compact in $L^1(X)$.
	\end{proposition}
	
	\begin{proof}
		By Theorem \ref{Uniform estimate}, the set $\mathcal{U}$ is uniformly bounded in $L^\infty$.
		Moreover, functions in $\mathcal{U}$ are $\omega$-$m$-subharmonic,
		hence satisfy uniform quasi-continuity properties.
		
		Therefore, by standard compactness results in pluripotential theory,
		$\mathcal{U}$ is relatively compact in $L^1(X)$.
	\end{proof}
	
	\subsection{Further remarks}
	
	The above results show that the pluripotential approach developed in this
	paper provides a robust framework for studying complex Hessian equations
	on Hermitian manifolds.
	
	In particular, the combination of weak comparison principles with torsion
	error and capacity techniques allows one to recover key features of the
	Kähler theory, including uniform estimates, existence, and stability,
	despite the lack of closedness of the background metric.

	\section*{Declarations}

	The author declares that there are no competing interests.

\end{document}